\documentclass{article}

\usepackage{amsmath, amsthm, amsfonts, amssymb}

\numberwithin{equation}{section}
\Large
\def\supp{\operatorname{supp}}

\def\a{\alpha}

\newtheorem{thm}{Theorem}[section]

\newtheorem{lem}{Lemma}[section]
\newtheorem{cor}{Corollary}[section]
\newtheorem{prop}{Proposition}[section]
\newtheorem{rem}{Remark}[section]
\numberwithin{equation}{section}

\title{Noncritical Weighted Hardy's Inequalities \\with compact perturbations}
\author{HIROSHI ANDO  and  TOSHIO HORIUCHI}
\date{}

\begin{document}
\maketitle
\begin{abstract}
Let  $\Omega $ be  a  bounded domain of $\mathbb R^N \ (N\ge 1)$ whose boundary  $\partial\Omega$ is a $C^2$ compact manifolds.
In the present paper  we  shall study  a variational problem relating the weighted  Hardy  inequalities with sharp missing terms established in \cite{AHL}. 
As weights  we adopted powers of the distance function $\delta(x)$ to the boundary $\partial\Omega$.
\footnote{Keywords: Weighted Hardy's inequalities, nonlinear variational problem, 
Weak  Hardy property, $p$-Laplace operator with weights,\\
2000 mathematics Subject Classification: Primary 35J70, Secondary 35J60, 34L30, 26D10.\\ 
This research was partially supported by Grant-in-Aid for Scientific Research (No.16K05189) 
and (No.15H03621).}
\end{abstract}

\section{ Introduction}

Let  $\Omega $ be  a  bounded domain of $\mathbb R^N \ (N\ge 1)$ whose boundary  $\partial\Omega$ is a $C^2$ compact manifolds.
In \cite{AHL} we  have established  $N$ dimensional  weighted Hardy's inequalities 
with weight function being  powers of the  distance function 
$\delta (x)={\rm dist}(x,\partial\Omega)$ to the   boundary $\partial\Omega$. 
In  this  paper we  shall study  a variational problem relating to these new inequalities. \par
We  prepare  more  notations to  describe  our  results. 
Let $1<p<\infty$ and $\alpha<1-1/p$. 
By $ L^p(\Omega,\delta^{\alpha p})$ 
we denote the space of Lebesgue measurable functions with weight $\delta^{\alpha p}$, 
for which 
\begin{equation} \label{L^p-norm} 
|| u ||_{ L^p(\Omega,\delta^{\alpha p})} 
= \left( \int_{\Omega}|u|^p \delta^{\alpha p}dx\right)^{\!1/p} < \infty. 
\end{equation} 
$W_{\alpha,0}^{1,p}(\Omega)$ is given by the  completion of $C_c^\infty(\Omega)$ 
with respect to the norm defined by
\begin{equation}\label{W-norm}
|| u||_{ W_{\alpha,0}^{1,p}(\Omega)} =
|| |\nabla u|  ||_{ L^p(\Omega, \delta ^{\alpha p})}. 
\end{equation}
Then $W^{1,p}_{\alpha,0}(\Omega) $  becomes a Banach space  
with the norm $|| \cdot||_{ W_{\alpha,0}^{1,p}(\Omega)}$. 
Under these preparation 
we recall  the  noncritical weighted Hardy inequality in \cite{AHL}. 
In particular,  we   have the simplest one:
\begin{equation}\label{HI}
\int_\Omega  |\nabla u|^p \delta^{\alpha p}dx 
\ge \mu \int_\Omega |u|^p \delta ^{(\alpha -1)p}dx 
\quad \text{for} \ \ u\in W^{1,p}_{\alpha,0}(\Omega),
\end{equation}
where $\mu$ is a positive  constant independent of $u$. 
If  $\alpha=0$ and  $p=2$, then  (\ref{HI})  is  a well-known Hardy's inequality and  valid 
for  a bounded domain $\Omega $  of  $\mathbb R^N$ with Lipschitz boundary 
 (c.f. \cite{BM}, \cite{D}, \cite{KO}, \cite{MMP}).
If  $\Omega$ is  convex  and  $ \alpha=0$, then (\ref{HI}) with $\mu = (1-1/p)^p$ holds 
for arbitrary $1<p<\infty$ (see \cite{MMP}, \cite{MS}).

The best possible $\mu$ in (\ref{HI}) is given by the quantity
\begin{equation}
\inf_{u\in W_{\alpha,0}^{1,p}(\Omega),u\neq 0} 
\frac{\int_\Omega |\nabla u|^p \delta^{\alpha p}dx}{\int_\Omega |u|^p \delta^{(\alpha -1)p}dx},
\end{equation}
which depends on $p,\alpha$ and $\Omega$.

In this paper we consider the following variational problem 
\begin{equation}
J^\alpha_\lambda 
= \inf_{u\in W_{\alpha,0}^{1,p}(\Omega),u\neq 0} \chi^\alpha_\lambda (u)  \label{J}
\end{equation}
where $\lambda\in\mathbb R$ and 
\begin{equation} 
\chi^\alpha_\lambda (u)
=\frac{\int_\Omega |\nabla u|^p \delta^{\alpha p}dx 
-\lambda \int_\Omega|u|^p \delta^{\alpha p}dx}{\int_\Omega |u|^p \delta ^{(\alpha -1)p}dx}.\label{chi}
\end{equation}
Note that $J_0^\alpha$ gives the best constant in (\ref{HI}). 
Clearly, the function $\lambda\mapsto J^\alpha_\lambda$ is non-increasing on $\mathbb R$ 
and $ J^\alpha_\lambda\to -\infty$ as $\lambda\to \infty$.

\begin{rem}\label{rem1.1}
It is worthy to remark that (\ref{HI}) is never valid in the critical case that $\alpha\ge 1-1/p$.  Nevertheless, we have established in this case a variant of weighted Hardy's inequalities 
in \cite{AHL} (cf. \cite{H}). 
In a coming paper \cite{AH2},
we shall treat general  weighted Hardy's inequalities with  compact perturbations
and study relating variational problems including    the critical case that $\alpha\ge 1-1/p$.
\end{rem}

This paper is organized in the following way: 
The main result is described in Section 2. 
Section 3 is devoted to the proof of main result.

\section{Main results}

Our main result is the following.

\begin{thm}\label{Main} 
Assume  that  $\Omega$  is  a  bounded domain of  class $C^2$ in $\mathbb R^N$. 
Assume  that $1<p<\infty$  and  $\alpha<1-1/p$. 
Then  there exists a constant $\lambda^\ast \in \mathbb R$ such that:
\begin{enumerate}
\item If  $\lambda \le  \lambda^\ast$,  then $J^\alpha_\lambda = \Lambda_{\alpha,p}$. 
If  $\lambda >  \lambda^\ast$,  then $J^\a_\lambda < \Lambda_{\a,p}$.
\end{enumerate}
Here 
\begin{equation}\label{Lambda}
\Lambda _{\alpha,p} = \left( 1-\alpha-\frac 1p\right)^{\!p}. 
\end{equation}
Moreover, it holds that:
\begin{enumerate}
\item[2.] If $\lambda<\lambda^\ast$, 
then the infimum $J_\lambda^\alpha$ in (\ref{J}) is not attained.
\item[3.] If $\lambda>\lambda^\ast$, 
then the infimum $J_\lambda^\alpha$ in (\ref{J}) is attained.
\end{enumerate}
\end{thm}

\begin{rem}\label{rem2.1}
\begin{enumerate}
\item In Theorem \ref{Main}, it remains for $\lambda=\lambda^\ast$ of the open problem 
whether the infimum $J_{\lambda^\ast}^\alpha$ in (\ref{J}) is attained or not.

\item For the case of $\alpha=0$ and $p=2$, it is shown that 
the infimum $J_\lambda^0$ in (\ref{J}) is attained if and only if $\lambda>\lambda^\ast$. 
See \cite{BM}.

\item For the case of $\alpha=0$ and $\lambda=0$, 
the value of the infimum $J_0^0$ in (\ref{J}) and its attainability are studied in \cite{MMP}.

\item In the assertion 3 of Theorem \ref{Main}, if $\lambda>\lambda^\ast$ then 
the minimizer $u$ for the variational problem (\ref{J}) is a non-trivial weak solution 
of the following Euler-Lagrange equation:
\begin{equation*}
-{\rm div}(\delta^{\alpha p}|\nabla u|^{p-2}\nabla u)-\lambda \delta^{\alpha p}|u|^{p-2}u
=J_\lambda^\alpha \delta^{(\alpha-1)p} |u|^{p-2}u \quad \text{in} \ \ {\cal D}'(\Omega).
\end{equation*}
\end{enumerate}
\end{rem}

\begin{cor}\label{cor2.1}
Under the  same assumptions  as in Theorem \ref{Main}, 
there exists a constant $\lambda \in \mathbb R$ 
such  that for $u\in W^{1,p}_{\alpha,0}(\Omega)$
\begin{equation}\label{HI1}
\int_\Omega  |\nabla u|^p \delta^{\alpha p}dx 
\ge \Lambda_{\alpha,p} \int_\Omega |u|^p \delta ^{(\alpha -1)p}dx 
+ \lambda \int_\Omega|u|^p \delta^{\alpha p}dx.
\end{equation}
\end{cor}

\medskip

For each  small $\eta>0$, 
by $\Omega_\eta $ we denote a tubular neighborhood of $\partial \Omega$;
\begin{equation}\label{t-nbd}
\Omega_\eta = \{ x\in \Omega:  \delta(x)= {\rm dist}(x,\partial\Omega)<\eta \}. 
\end{equation}
Then we have the following inequality of Hardy type which  is crucial 
in the proof of Theorem \ref{Main}.

\begin{thm}\label{thm2.2}
Assume  that  $\Omega$  is  a  bounded domain of  class $C^2$ in $\mathbb R^N$. 
Assume  that $1<p<\infty$  and  $ \a<1-1/p$. 
Assume  that  $\eta $  is  a sufficienty  small positive  number. 
Then we have that for $u\in W^{1,p}_{\alpha,0}(\Omega)$
\begin{equation}\label{HI2}
\int_{\Omega_\eta}  |\nabla u|^p \delta^{\alpha p} dx
\ge \Lambda_{\alpha,p} \int_{\Omega_\eta}  |u|^p \delta ^{(\alpha -1)p}dx,
\end{equation}
where $\Lambda_{\alpha,p}$ is defined by (\ref{Lambda}).
\end{thm}

In \cite{AHL} we have more precise estimate than (\ref{HI2}).

\begin{cor}\label{cor2.2}
Under the  same assumptions  as in Theorem \ref{thm2.2}, 
there exists a positive constant $\gamma$ such  that for $u\in W^{1,p}_{\alpha,0}(\Omega)$
\begin{equation}\label{HI3}
\int_\Omega  |\nabla u|^p \delta^{\alpha p}dx 
\ge \gamma \int_\Omega |u|^p \delta ^{(\alpha -1)p}dx.
\end{equation}
\end{cor}

For  any  bounded  domain $\Omega \subset \mathbb R^N$ we can prove  the following:

\begin{thm}\label{thm2.3}
Assume  that  $\Omega$  is  a  bounded domain of   $\mathbb R^N$.
Assume  that $1<p<\infty$  and  $ \alpha<1-1/p$.  
Then the followings are equivalent  with each other.
\begin{enumerate}
\item  There exists a  positive  number  $\gamma$  such  that the  inequality (\ref{HI3}) is valid for  every  $u\in W^{1,p}_{\alpha,0}(\Omega)$.

\item  For a sufficiently small positive number $\eta$,
there exists a positive number $\kappa$ such that the  inequality (\ref{HI2})  with  $\Lambda_{\alpha,p}$ replaced by $\kappa$ is valid for every  $u\in W^{1,p}_{\alpha,0}(\Omega)$.
\end{enumerate} 
\end{thm}

\medskip

For the proofs of Theorem \ref{thm2.2}, Corollary \ref{cor2.2} and Theorem \ref{thm2.3}, 
see in \cite{AHL}.

\section{Proof of Theorem {\ref{Main}}}

In this section, we give the proof of Theorem \ref{Main}.

\subsection{Upper bound of  $J_\lambda^\alpha$}

First, we prove the assertion 1 of Theorem \ref{Main}.

\begin{lem}\label{lem3.1}
Let $1<p<\infty$ and $\alpha<1-1/p$. 
For any $\varepsilon>0$ and any $\eta>0$ 
there exists a function $h\in W^{1,p}_{\alpha, 0}((0,\eta))$ such that 
\begin{equation}\label{3.1}
\int_0^\eta |h'(t)|^p t^{\alpha p}dt 
\le (\Lambda_{\alpha,p}+\varepsilon)\int_0^\eta |h(t)|^pt^{(\alpha-1)p}dt, 
\end{equation}
where $\Lambda_{\alpha,p}$ is defined by (\ref{Lambda}).
\end{lem}

\begin{proof}
Since the inequality (\ref{3.1}) is invariant with respect to scaling, 
we may assume that $\eta=2$. 
Put 
\begin{equation*}
h(t)= \begin{cases} \ \ t^\beta & \text{if} \quad t\in (0,1),\\ 2-t & \text{if} \quad t\in [1,2) \end{cases}  
\end{equation*}
with $\beta>1-\alpha-1/p$. Then we see that $h\in W^{1,p}_{\alpha,0}((0,2))$,
\begin{equation}\label{3.2}
\int_0^2  |h'(t)|^p t^{\alpha p}dt = \frac{\beta^p}{p(\beta-1+\alpha+1/p)} +C_{\alpha,p}  
\end{equation}
and 
\begin{equation}\label{3.3}
\int_0^2 |h(t)|^pt^{(\alpha-1)p}dt = \frac{1}{p(\beta-1+\alpha+1/p)} +D_{\alpha,p},  
\end{equation}
where 
\begin{equation*}
C_{\alpha,p} = \int_1^2t^{\alpha p}dt \quad \text{and} \quad 
D_{\alpha,p} = \int_1^2 (2-t)^pt^{(\alpha-1)p}dt
\end{equation*}
are constants independent of $\beta$. 
It follows from (\ref{3.2}) and (\ref{3.3}) that  
\begin{equation*}
\frac{\int_0^2  |h'(t)|^p t^{\alpha p}dt}{\int_0^2 |h(t)|^pt^{(\alpha-1)p}dt} 
\longrightarrow \Lambda_{\alpha,p} 
\quad \text{as} \ \ \beta \to 1-\alpha-\frac{1}{p}+0,
\end{equation*}
which implies (\ref{3.1}) with $\eta=2$. Therefore we obtain the desired conclusion.
\end{proof}

\begin{lem}\label{lem3.2}
Let $\Omega$ be  a  bounded domain of  class $C^2$ in $\mathbb R^N$. 
Let $1<p<\infty$ and $\alpha<1-1/p$. 
Then it holds that 
\begin{equation}\label{ubd} 
J_\lambda^\alpha \le \Lambda_{\alpha,p} 
\end{equation}
for all $\lambda\in \mathbb R$.
\end{lem}

\begin{proof}
Since the boundary $\partial\Omega$ is of class $C^2$, there exists an $\eta_0>0$ such that 
for any $\eta\in(0,\eta_0)$ and every $x\in\Omega_\eta$ 
we have a unique point $\sigma(x)\in\partial\Omega$ satisfying $\delta(x)=|x-\sigma(x)|$. 
The mapping 
\begin{equation*}
\Omega_\eta\ni x\mapsto (\delta(x),\sigma(x))=(t,\sigma)\in(0,\eta)\times\partial\Omega
\end{equation*}
is a $C^2$ diffeomorphism, and its inverse is given by 
\begin{equation*}
(0,\eta)\times\partial\Omega\ni(t,\sigma)\mapsto 
x(t,\sigma)=\sigma+t\cdot n(\sigma)\in\Omega_\eta,
\end{equation*}
where $n(\sigma)$ is the inward unit normal to $\partial\Omega$ at $\sigma\in\partial\Omega$. 
For each $t\in(0,\eta)$, the mapping 
\begin{equation*}
\partial\Omega\ni\sigma\mapsto 
\sigma_t(\sigma)=x(t,\sigma)\in\Sigma_t=\{x\in\Omega:\delta(x)=t\}
\end{equation*}
is a also a $C^2$ diffeomorphism of $\partial\Omega$ onto $\Sigma_t$, 
and its Jacobian satisfies 
\begin{equation}\label{3.5}
|\text{Jac}\,\sigma_t(\sigma)-1|\le ct \qquad \text{for any} \ \ \sigma\in\partial\Omega,
\end{equation}
where $c$ is a positive constant depending only on $\eta_0$, $\partial\Omega$  
and the choice of local coordinates. 
Since $n(\sigma)$ is orthogonal to $\Sigma_t$ 
at $\sigma_t(\sigma)=\sigma+t\cdot n(\sigma)\in\Sigma_t$, 
it follows that for every integrable function $v$ in $\Omega_\eta$
\begin{align}\label{3.6}
\int_{\Omega_\eta}v(x)dx 
& =\int_0^\eta dt\int_{\Sigma_t}v(\sigma_t)d\sigma_t 
\nonumber \\
& =\int_0^\eta dt\int_{\partial\Omega}v(x(t,\sigma))|\text{Jac}\,\sigma_t(\sigma)|d\sigma,
\end{align}
where $d\sigma$ and $d\sigma_t$ denote surface elements on 
$\partial\Omega$ and $\Sigma_t$, respectively. 
Hence (\ref{3.6}) together with (\ref{3.5}) implies that 
for every integrable function $v$ in $\Omega_\eta$
\begin{align}
\int_0^\eta (1-ct)dt\int_{\partial\Omega}|v(x(t,\sigma))|d\sigma 
& \le \int_{\Omega_\eta}|v(x)|dx \label{3.7} 
\\ 
& \le \int_0^\eta (1+ct)dt\int_{\partial\Omega}|v(x(t,\sigma))|d\sigma. \label{3.8}
\end{align}

Let $\varepsilon>0$, and let $\eta\in(0,\eta_0)$. 
Take $h\in W^{1,p}_{\alpha, 0}((0,\eta))$ be a function satisfying (\ref{3.1}). 
Put 
\begin{equation}\label{3.9}
u(x)= \begin{cases} h(\delta(x)) & \text{if} \quad x\in \Omega_\eta,\\ 
\quad \ 0 & \text{if} \quad x\in \Omega\setminus\Omega_\eta. \end{cases}  
\end{equation}
Since $|\nabla u(x)|=|h'(\delta(x))|$ for $x\in\Omega_\eta$ by $|\nabla \delta(x)|=1$, 
it follows from (\ref{3.8}) that 
\begin{equation}\label{3.10}
\int_{\Omega_\eta}|\nabla u|^p\delta^{\alpha p}dx 
\le (1+c\eta)|\partial\Omega|\int_0^\eta |h'(t)|^pt^{\alpha p}dt,  
\end{equation}
which implies $u\in W^{1,p}_{\alpha,0}(\Omega)$ by $\supp u\subset \Omega_\eta$. 
On the other hand, by (\ref{3.7}) and (\ref{3.9}) we have that 
\begin{equation}\label{3.11}
\int_{\Omega_\eta}|u|^p\delta^{(\alpha-1) p}dx 
\ge (1-c\eta)|\partial\Omega|\int_0^\eta |h(t)|^pt^{(\alpha-1) p}dt.  
\end{equation}
Since $\supp u\subset \Omega_\eta$, 
by combining (\ref{3.10}), (\ref{3.11}) and the estimate 
\begin{equation*}
\int_{\Omega_\eta}|u|^p\delta^{\alpha p}dx 
\le \eta^p\int_{\Omega_\eta}|u|^p\delta^{(\alpha-1) p}dx,  
\end{equation*}
we obtain that 
\begin{equation*}
\chi_\lambda^\alpha(u) 
\le \frac{1+c\eta}{1-c\eta}\frac{\int_0^\eta |h'(t)|^pt^{\alpha p}dt}
{\int_0^\eta |h(t)|^pt^{(\alpha-1) p}dt} +|\lambda|\eta^p. 
\end{equation*}
This together with (\ref{3.1}) implies that 
\begin{equation}\label{3.12}
J_\lambda^\alpha 
\le \frac{1+c\eta}{1-c\eta}(\Lambda_{\alpha,p}+\varepsilon) +|\lambda|\eta^p. 
\end{equation}
Letting $\eta\to +0$ in (\ref{3.12}), (\ref{ubd}) follows. Therefore it concludes the proof.
\end{proof}

\begin{lem}\label{lem3.3}
Let $\Omega$ be  a  bounded domain of  class $C^2$ in $\mathbb R^N$. 
Let $1<p<\infty$ and $\alpha<1-1/p$. 
Then there exists a $\lambda\in \mathbb R$ such that $J_\lambda^\alpha =\Lambda_{\alpha,p}$.
\end{lem}

\begin{proof}
Let $\eta>0$ be a sufficiently small number as in Theorem \ref{thm2.2}. 
For any $u\in W^{1,p}_{\alpha,0}(\Omega)\setminus\{0\}$, 
by using Hardy's inequality (\ref{HI2}) and the estimate 
\begin{equation*}
\int_{\Omega\setminus\Omega_\eta}|u|^p\delta^{(\alpha-1) p}dx
\le \eta^{-p}\int_{\Omega}|u|^p\delta^{\alpha p}dx,
\end{equation*}
we have that 
\begin{align*}
\Lambda_{\alpha,p}\int_{\Omega}|u|^p\delta^{(\alpha-1) p}dx 
& = \Lambda_{\alpha,p}\int_{\Omega_\eta}|u|^p\delta^{(\alpha-1) p}dx
+\Lambda_{\alpha,p}\int_{\Omega\setminus\Omega_\eta}|u|^p\delta^{(\alpha-1) p}dx 
\\ 
& \le \int_{\Omega}|\nabla u|^p\delta^{\alpha p}dx  
+ \Lambda_{\alpha,p}\eta^{-p}\int_{\Omega}|u|^p\delta^{\alpha p}dx,
\end{align*}
which implies that 
\begin{equation*}
\chi_\lambda^\alpha(u)\ge \Lambda_{\alpha,p}
\end{equation*}
for $\lambda\le -\Lambda_{\alpha,p}\eta^{-p}$. 
Consequently, it holds that $J_\lambda^\alpha \ge \Lambda_{\alpha,p}$ 
for $\lambda\le -\Lambda_{\alpha,p}\eta^{-p}$.  
This together with (\ref{ubd}) implies the desired conclusion. 
\end{proof}

\begin{lem}\label{lem3.4}
Let $\Omega$ be  a  bounded domain of  class $C^2$ in $\mathbb R^N$. 
Let $1<p<\infty$ and $\alpha<1-1/p$. 
Then the function $\lambda\mapsto J^\alpha_\lambda$ is Lipschitz continuous on $\mathbb R$.
\end{lem}

\begin{proof}
Let $\lambda,\bar{\lambda}\in\mathbb R$. 
Then it holds that for any $u \in W^{1,p}_{\alpha,0}(\Omega)\setminus\{0\}$
\begin{equation*}
|\chi_\lambda^\alpha(u)-\chi_{\bar{\lambda}}^\alpha(u)|
=|\lambda -\bar{\lambda}|
\frac{\int_\Omega|u|^p \delta^{\alpha p}dx}{\int_\Omega|u|^p \delta^{(\alpha-1)p}dx}
\le M^p|\lambda -\bar{\lambda}|,
\end{equation*}
where $M=\sup_{x\in\Omega}\delta(x)$ is a positive constant depending only on $\Omega$. 
Hence we see that  
\begin{equation*}
|J_\lambda^\alpha-J_{\bar{\lambda}}^\alpha|\le M^p|\lambda -\bar{\lambda}|
\end{equation*}
for $\lambda,\bar{\lambda}\in\mathbb R$. It completes the proof.
\end{proof}

\medskip

\noindent
{\bf  Proof of the assertion 1 of Theorem \ref{Main}.} \ 
By Lemma \ref{lem3.3} and $\lim_{\lambda\to\infty}J_\lambda^\alpha=-\infty$,  
the set $\{\lambda \in \mathbb R : J_\lambda^\alpha = \Lambda_{\alpha,p} \}$ 
is non-empty and upper bounded. 
Hence the 
$\sup\{\lambda \in \mathbb R : J_\lambda^\alpha = \Lambda_{\alpha,p} \}$ exists finitely. 
Put 
\begin{equation}
\lambda^\ast = \sup\{\lambda \in \mathbb R : J_\lambda^\alpha = \Lambda_{\alpha,p} \}. \label{lambda*}
\end{equation}
Since the function $\lambda\mapsto J_\lambda^\alpha$ is non-increasing on $\mathbb R$, 
it follows from Lemma \ref{lem3.2} and Lemma \ref{lem3.3} that 
$J_\lambda^\alpha=\Lambda_{\alpha,p}$ for $\lambda<\lambda^\ast$ 
and $J_\lambda^\alpha<\Lambda_{\alpha,p}$ for $\lambda>\lambda^\ast$. 
Further, by Lemma \ref{lem3.4} 
we have the equality $J_{\lambda^\ast}^\alpha=\Lambda_{\alpha,p}$.
Therefore the assertion 1 of Theorem \ref{Main} is valid.
\qed

\subsection{ $J^\alpha_\lambda$ is  not  attained when $\lambda < \lambda^\ast$}

Next, we prove the assertion 2 of Theorem \ref{Main}.

\medskip

\noindent
{\bf  Proof of the assertion 2 of Theorem \ref{Main}.} \ 
Suppose that for some $\lambda<\lambda^\ast$ 
the infimum $J_\lambda^\alpha$ in (\ref{J}) is attained 
at an element $u\in W^{1,p}_{\alpha,0}(\Omega)\setminus\{0\}$. 
Then, by the assertion 1 of Theorem \ref{Main}, we have that 
\begin{equation}\label{3.14}
\chi_\lambda^\alpha(u) =J_\lambda^\alpha= \Lambda_{\alpha,p}
\end{equation}
and for $\lambda<\bar{\lambda}<\lambda^\ast$
\begin{equation}\label{3.15}
\chi_{\bar{\lambda}}^\alpha(u) \ge J_{\bar{\lambda}}^\alpha= \Lambda_{\alpha,p}. 
\end{equation}
From (\ref{3.14}) and (\ref{3.15}) it follows that 
\begin{equation*}
(\bar{\lambda}-\lambda)\int_{\Omega}|u|^p\delta^{\alpha p}dx \le 0. 
\end{equation*}
Since $\bar{\lambda}-\lambda>0$, we conclude that 
\begin{equation*}
\int_{\Omega}|u|^p\delta^{\alpha p}dx = 0,
\end{equation*}
which contradicts $u\ne 0$ in $W^{1,p}_{\alpha,0}(\Omega)$. 
Therefore it completes the proof.
\qed

\subsection{Attainability of $J^\alpha_\lambda$ when $\lambda>\lambda^\ast$ }

At last, we prove the assertion 3 of Theorem \ref{Main}.

\medskip

Let $\{u_k\}$ be a minimizing sequence for the variational problem (\ref{J}) normalized so that 
\begin{equation}\label{ms1}
\int_{\Omega}|u_k|^p\delta^{(\alpha-1) p}dx = 1 \quad \text{for all} \ k.  
\end{equation}
Since $\{u_k\}$ is bounded in $W^{1,p}_{\alpha,0}(\Omega)$, by taking a suitable subsequence, 
we may assume that there exists a $u\in W^{1,p}_{\alpha,0}(\Omega)$ such that 
\begin{align}
\nabla u_k \stackrel{weak}{\longrightarrow} \nabla u \quad 
& \text{in} \ \  (L^p(\Omega,\delta^{\alpha p}))^N, \label{ms2} \\
u_k  \stackrel{weak}{\longrightarrow} u \quad 
& \text{in} \ \  L^p(\Omega,\delta^{(\alpha-1) p}) \label{ms3} 
\end{align}
and 
\begin{equation}\label{ms4} 
u_k \longrightarrow u \quad \text{in} \ \  L^p(\Omega,\delta^{\alpha p}) 
\end{equation}
by Hardy's inequality (\ref{HI3}) and the compact embedding 
$W^{1,p}_{\alpha,0}(\Omega) \hookrightarrow L^p(\Omega,\delta^{\alpha p})$. 

\medskip

Under these preparation we establish the properties of concentration and compactness 
for the minimizing sequence, respectively.

\begin{prop}\label{concentration}
Let $\Omega$ be  a  bounded domain of  class $C^2$ in $\mathbb R^N$. 
Let $1<p<\infty$ and $\alpha<1-1/p$. Let $\lambda\in \mathbb R$. 
Let $\{u_k\}$ be a minimizing sequence for (\ref{J}) satisfying (\ref{ms1}), (\ref{ms2}), (\ref{ms3}) and (\ref{ms4}) with $u=0$. 
Then it holds that 
\begin{equation}\label{3.20}
\nabla u_k \longrightarrow 0 \quad \text{in} \ \ (L^p_{\rm loc}(\Omega))^N 
\end{equation}
and 
\begin{equation}\label{3.21}
J_\lambda^\alpha = \Lambda_{\alpha,p}.  
\end{equation}
\end{prop}

\begin{proof} 
Let $\eta>0$ be a sufficiently small number as in Theorem \ref{thm2.2}. 
By Hardy's inequality (\ref{HI2}) and (\ref{ms1}) we have that 
\begin{align}
\int_{\Omega_\eta}|\nabla u_k|^p\delta^{\alpha p}dx 
& \ge \Lambda_{\alpha,p}\int_{\Omega_\eta}|u_k|^p\delta^{(\alpha-1) p}dx \nonumber
\\ 
& =\Lambda_{\alpha,p}
\left(1-\int_{\Omega\setminus\Omega_\eta}|u_k|^p\delta^{(\alpha-1) p}dx\right), \nonumber 
\end{align}
and so 
\begin{align}\label{3.22}
\chi_\lambda^\alpha(u_k) \ge 
& \Lambda_{\alpha,p}
\left(1-\int_{\Omega\setminus\Omega_\eta}|u_k|^p\delta^{(\alpha-1) p}dx\right) \nonumber 
\\ 
& + \int_{\Omega\setminus\Omega_\eta}|\nabla u_k|^p\delta^{\alpha p}dx 
-\lambda \int_\Omega |u_k|^p\delta^{\alpha p}dx. 
\end{align}
Since 
\begin{equation*}
\int_{\Omega\setminus\Omega_\eta}|u_k|^p\delta^{(\alpha-1) p}dx 
\le \eta^{-p} \int_{\Omega}|u_k|^p\delta^{\alpha p}dx,
\end{equation*}
it follows from (\ref{ms4}) with $u=0$ that 
\begin{equation}\label{3.23}
\lim_{k\to\infty}\int_{\Omega\setminus\Omega_\eta}|u_k|^p\delta^{(\alpha-1) p}dx=0.  
\end{equation}
Hence, by (\ref{3.22}), (\ref{3.23}) and (\ref{ms4}) with $u=0$, we obtain that 
\begin{equation*}
\limsup_{k\to\infty}\int_{\Omega\setminus\Omega_\eta}|\nabla u_k|^p\delta^{\alpha p}dx 
\le J_\lambda^\alpha -\Lambda_{\alpha,p}. 
\end{equation*}
Since $J_\lambda^\alpha -\Lambda_{\alpha,p}\le 0$ by Lemma \ref{3.2}, we conclude that 
\begin{equation}\label{3.24}
\lim_{k\to\infty}\int_{\Omega\setminus\Omega_\eta}|\nabla u_k|^p\delta^{\alpha p}dx=0, 
\end{equation}
and so 
\begin{equation*}
\lim_{k\to\infty}\int_{\Omega\setminus\Omega_\eta}|\nabla u_k|^p dx=0. 
\end{equation*}
This shows (\ref{3.20}). 
Moreover, letting $k\to\infty$ in (\ref{3.22}), 
it follows from (\ref{3.23}), (\ref{3.24}) and (\ref{ms4}) with $u=0$ that 
\begin{equation*}
J_\lambda^\alpha \ge \Lambda_{\alpha,p}.
\end{equation*}
This together with Lemma \ref{3.2} implies (\ref{3.21}). 
Consequently it completes the proof.
\end{proof}

\begin{prop}\label{compactness}
Let $\Omega$ be  a  bounded domain of  class $C^2$ in $\mathbb R^N$. 
Let $1<p<\infty$ and $\alpha<1-1/p$. Let $\lambda\in \mathbb R$. 
Let $\{u_k\}$ be a minimizing sequence for (\ref{J}) satisfying (\ref{ms1}), (\ref{ms2}), (\ref{ms3}) and (\ref{ms4}) with $u\ne 0$. 
Then it holds that 
\begin{equation}\label{3.25}
J_\lambda^\alpha =\min(\Lambda_{\alpha,p}, \chi_\lambda^\alpha(u)). 
\end{equation}
In addition, if $J_\lambda^\alpha <\Lambda_{\alpha,p}$, then it holds that 
\begin{equation}\label{minimizer} 
J_\lambda^\alpha=\chi_\lambda^\alpha(u), 
\end{equation}
namely $u$ is a minimizer for (\ref{J}), and  
\begin{equation}\label{3.27}
u_k \longrightarrow u \quad \text{in} \ \ W^{1,p}_{\alpha,0}(\Omega). 
\end{equation} 
\end{prop}

\begin{proof} 
Let $\eta>0$ be a sufficiently small number as in Theorem \ref{thm2.2}. 
Then we have (\ref{3.22}) by the same arguments as in the proof of Proposition 
\ref{concentration}. 
By the estimate 
\begin{equation*}
\int_{\Omega\setminus\Omega_\eta}|u_k-u|^p\delta^{(\alpha-1) p}dx 
\le \eta^{-p} \int_{\Omega}|u_k-u|^p\delta^{\alpha p}dx,
\end{equation*}
(\ref{ms4}) implies that 
\begin{equation}\label{3.28} 
\lim_{k\to\infty}\int_{\Omega\setminus\Omega_\eta}|u_k|^p\delta^{(\alpha-1) p}dx 
= \int_{\Omega\setminus\Omega_\eta}|u|^p\delta^{(\alpha-1) p}dx. 
\end{equation}
Since it follows from (\ref{ms2}) that $\nabla u_k \longrightarrow \nabla u$ 
weakly in $(L^p(\Omega\setminus\Omega_\eta,\delta^{\alpha p}))^N$, 
by weakly lower semi-continuity of the $L^p$-norm, we see that  
\begin{align}\label{3.29} 
\liminf_{k\to\infty} \int_{\Omega\setminus\Omega_\eta}|\nabla u_k|^p\delta^{\alpha p}dx 
& \ge \left(\liminf_{k\to\infty} 
\| |\nabla u_k| \|_{L^p(\Omega\setminus\Omega_\eta,\delta^{\alpha p})}\right)^p \nonumber 
\\ 
& \ge \| |\nabla u| \|_{L^p(\Omega\setminus\Omega_\eta,\delta^{\alpha p})}^p \nonumber 
\\ 
& =\int_{\Omega\setminus\Omega_\eta}|\nabla u|^p\delta^{\alpha p}dx.  
\end{align}
Hence, by letting $k\to\infty$ in (\ref{3.22}), 
from (\ref{ms4}), (\ref{3.28}) and (\ref{3.29}) it follows that 
\begin{align}\label{3.30}
J_\lambda^\alpha \ge 
\Lambda_{\alpha,p} 
& \left(1-\int_{\Omega\setminus\Omega_\eta}|u|^p\delta^{(\alpha-1) p}dx\right) \nonumber 
\\ 
& + \int_{\Omega\setminus\Omega_\eta}|\nabla u|^p\delta^{\alpha p}dx 
-\lambda \int_\Omega |u|^p\delta^{\alpha p}dx. 
\end{align}
Letting $\eta\to +0$ in (\ref{3.30}), we obtain that 
\begin{align}\label{3.31}
J_\lambda^\alpha \ge \Lambda_{\alpha,p} 
& \left(1-\int_{\Omega}|u|^p\delta^{(\alpha-1) p}dx\right) \nonumber 
\\ 
& + \int_{\Omega}|\nabla u|^p\delta^{\alpha p}dx 
-\lambda \int_\Omega |u|^p\delta^{\alpha p}dx. 
\end{align}
Since it holds that 
\begin{equation}\label{3.32}
0<\int_\Omega |u|^p\delta^{(\alpha-1)p}dx \le 
\liminf_{k\to\infty}\int_\Omega |u_k|^p\delta^{(\alpha-1)p}dx =1 
\end{equation}
by $u\ne 0$, (\ref{ms1}), (\ref{ms3}) and weakly lower semi-continuity of the $L^p$-norm, 
we have from (\ref{3.31}) and (\ref{3.32}) that 
\begin{align}\label{3.33} 
J_\lambda^\alpha 
& \ge \Lambda_{\alpha,p} \left(1-\int_{\Omega}|u|^p\delta^{(\alpha-1) p}dx\right) 
+ \chi_\lambda^\alpha(u) \int_{\Omega}|u|^p\delta^{(\alpha-1)p}dx \nonumber  
\\ 
& \ge \min(\Lambda_{\alpha,p}, \chi_\lambda^\alpha(u)). 
\end{align}
This together with Lemma \ref{3.2} implies (\ref{3.25}). 
Moreover, by (\ref{3.25}) and (\ref{3.33}), we conclude that 
\begin{equation}\label{3.34}
J_\lambda^\alpha 
= \Lambda_{\alpha,p} \left(1-\int_{\Omega}|u|^p\delta^{(\alpha-1) p}dx\right) 
+ \chi_\lambda^\alpha(u) \int_{\Omega}|u|^p\delta^{(\alpha-1)p}dx.  
\end{equation}
In addition, if $J_\lambda^\alpha <\Lambda_{\alpha,p}$, 
then $J_\lambda^\alpha=\chi_\lambda^\alpha(u)$ by (\ref{3.25}), 
and so, it follows from (\ref{3.34}) and (\ref{ms1}) that 
\begin{equation}\label{3.35} 
\int_{\Omega}|u|^p\delta^{(\alpha-1) p}dx=1
=\lim_{k\to\infty}\int_{\Omega}|u_k|^p\delta^{(\alpha-1) p}dx. 
\end{equation}
(\ref{ms3}) and (\ref{3.35}) imply that  
\begin{equation}
u_k \longrightarrow u \quad \text{in} \ \ L^p(\Omega,\delta^{(\alpha-1)p}). 
\end{equation}
Further, by (\ref{ms1}), (\ref{ms4}), (\ref{minimizer}) and (\ref{3.35}), we obtain that 
\begin{align*}
\int_{\Omega}|\nabla u_k|^p\delta^{\alpha p}dx
& =  \chi_\lambda^\alpha(u_k)+\lambda \int_\Omega |u_k|^p\delta^{\alpha p}dx 
\\ 
& \longrightarrow 
\chi_\lambda^\alpha(u)+\lambda \int_\Omega |u|^p\delta^{\alpha p}dx 
=\int_{\Omega}|\nabla u|^p\delta^{\alpha p}dx. 
\end{align*}
This together with (\ref{ms2}) implies that 
\begin{equation}
\nabla u_k \longrightarrow \nabla u \quad \text{in} \ \ (L^p(\Omega,\delta^{\alpha p}))^N, 
\end{equation}
which shows (\ref{3.27}). Consequently it completes the proof.
\end{proof}

\medskip

\noindent
{\bf  Proof of the assertion 3 of Theorem \ref{Main}.} 
Let $\lambda>\lambda^\ast$. 
Then $J_\lambda^\alpha < \Lambda_{\alpha,p}$ by the assertion 1 of Theorem \ref{Main}. 
Let $\{u_k\}$ be a minimizing sequence for (\ref{J}) satisfying (\ref{ms1}), (\ref{ms2}), (\ref{ms3}) and (\ref{ms4}). 
Then we see that $u\ne 0$ by Proposition \ref{concentration}. 
Therefore, by applying Proposition \ref{compactness}, 
we conclude that $\chi_\lambda^\alpha(u) =J_\lambda^\alpha$, 
namely $u$ is a minimizer for (\ref{J}). 
It finishes the proof.
\qed

\bigskip\bigskip

\noindent
Hiroshi Ando:\\
Department of Mathematics, Faculty of Science, Ibaraki University,\\
Mito, Ibaraki, 310-8512, Japan;\\
hiroshi.ando.math@vc.ibaraki.ac.jp

\medskip

\noindent
Toshio Horiuchi:\\
Department of Mathematics, Faculty of Science, Ibaraki University,\\
Mito, Ibaraki, 310-8512, Japan;\\
toshio.horiuchi.math@vc.ibaraki.ac.jp

\end{document}